\documentclass[a4paper,UKenglish,cleveref,autoref,thm-restate]{lipics-v2021}

\usepackage{graphicx}
\usepackage{amsmath}
\usepackage{amsfonts}
\usepackage{amssymb}

\usepackage{csquotes}

\usepackage{enumerate}
\usepackage{enumitem}
\usepackage{cite}
\usepackage{tabularx}
\usepackage{booktabs}

\newtheorem{question}[theorem]{Question}

\Crefname{claim}{Claim}{Claims}

\DeclareMathOperator{\incident}{Inc}

\newcommand{\euler}{\mathrm{e}}
\newcommand{\Rbl}{\mathcal{R}_{\text{BL}}}
\newcommand{\Rtl}{\mathcal{R}_{\text{TL}}}
\newcommand{\Ruh}{\mathcal{R}_{\text{UH}}}
\newcommand{\Rst}{\mathcal{R}_{\text{ST}}}
\newcommand{\Rbb}{\ensuremath{\mathbb{R}}}
\newcommand{\Zbb}{\ensuremath{\mathbb{Z}}}
\newcommand{\Rcal}{\ensuremath{\mathcal{R}}}
\newcommand{\Hcal}{\ensuremath{\mathcal{H}}}

\title{Polychromatic Colorings of Geometric Hypergraphs via Shallow Hitting Sets}

\author{Tim Planken}{University of Birmingham}{txp265@student.bham.ac.uk}{}{}
\author{Torsten Ueckerdt}{Karlsruhe Institute of Technology}{torsten.ueckerdt@kit.edu}{https://orcid.org/0000-0002-0645-9715}{}

\authorrunning{T. Planken and T. Ueckerdt}

\Copyright{T. Planken and T. Ueckerdt}



\hideLIPIcs
\nolinenumbers

\begin{document}

\maketitle

\begin{abstract}
    A range family $\Rcal$ is a family of subsets of $\Rbb^d$, like all halfplanes, or all unit disks.
    Given a range family $\Rcal$, we consider the $m$-uniform range capturing hypergraphs $\Hcal(V,\Rcal,m)$ whose vertex-sets $V$ are finite sets of points in $\Rbb^d$ with any $m$ vertices forming a hyperedge $e$ whenever $e = V \cap R$ for some $R \in \Rcal$.
    Given additionally an integer $k \geq 2$, we seek to find the minimum $m = m_\Rcal(k)$ such that every $\Hcal(V,\Rcal,m)$ admits a polychromatic $k$-coloring of its vertices, that is, where every hyperedge contains at least one point of each color.
    Clearly, $m_\Rcal(k) \geq k$ and the gold standard is an upper bound $m_{\Rcal}(k) = O(k)$ that is linear in $k$.

    A $t$-shallow hitting set in $\Hcal(V,\Rcal,m)$ is a subset $S \subseteq V$ such that $1 \leq |e \cap S| \leq t$ for each hyperedge $e$; i.e., every hyperedge is hit at least once but at most $t$ times by $S$.
    We show for several range families $\Rcal$ the existence of $t$-shallow hitting sets in every $\Hcal(V,\Rcal,m)$ with $t$ being a constant only depending on $\Rcal$.
    This in particular proves that $m_{\Rcal}(k) \leq tk = O(k)$ in such cases, improving previous polynomial bounds in $k$.
    Particularly, we prove this for the range families of all axis-aligned strips in $\Rbb^d$, all bottomless and topless rectangles in $\Rbb^2$, and for all unit-height axis-aligned rectangles in $\Rbb^2$.
\end{abstract}

\section{Introduction}
\label{sec:introduction}

We investigate polychromatic colorings of geometric hypergraphs defined by a finite set of points $V \subset \Rbb^d$ and a family $\Rcal$ of subsets of $\Rbb^d$, called a \emph{range family}.
Possible range families include for example all unit balls, all axis-aligned boxes, all halfplanes, or all translates of a fixed polygon.
In this paper we prove results for the following range families:
\begin{itemize}
    \itemsep0pt
    \item the family $\Rst = \Rst^1 \cup \cdots \cup \Rst^d$ with
    $\Rst^i = \{ \{(x_1,\ldots,x_d) \in \Rbb^d \mid a \leq x_i \leq b\} \mid a,b \in 
    \Rbb\}$ of all axis-aligned strips in $\Rbb^d$,
    \item the family $\Rbl = \{ [a,b] \times (-\infty,c] \mid a,b,c \in \Rbb \}$ of all bottomless rectangles in $\Rbb^2$,
    \item the family $\Rtl = \{ [a,b] \times [c, \infty) \mid a,b,c \in \Rbb \}$ of all topless rectangles in $\Rbb^2$, and
    \item the family $\Ruh = \{ [a,b] \times [c,c+1] \mid a,b,c \in \Rbb \}$ of all unit-height rectangles in $\Rbb^2$.
\end{itemize}

For a fixed range family $\Rcal$ and any finite point set $V \subset \Rbb^d$, the corresponding \emph{range capturing hypergraph} $H=\Hcal(V,\Rcal)$ has vertex set $V(H)=V$, and a subset $e \subseteq V$ is a hyperedge in $E(H)$ whenever there exists a range $R \in \Rcal$ with $e = V \cap R$.
In this case, we say that $e$ is \emph{captured} by the range $R$.
That is, we have points in $\Rbb^d$ and a subset of points forms a hyperedge whenever these vertices and no other vertices are captured by a range.
For example, a set $e$ of points in $V \subset \Rbb^d$ forms a hyperedge in $\Hcal(V,\Rst)$ if and only if in at least one of the $d$ coordinates, the points in $e$ are consecutive in $V$.
(We assume throughout that points in $V$ lie in general position, i.e., have pairwise different coordinates.)

For a positive integer $k$, a $k$-coloring $c \colon V \to \{1,\ldots,k\}$ of the vertices of a hypergraph $H=(V,E)$ is called \emph{proper} if each hyperedge $e \in E$ contains at least two colors, i.e., $|\{c(v) \mid v \in e\}| \geq 2$, and \emph{polychromatic} if each hyperedge $e \in E$ contains all $k$ colors, i.e., $|\{c(v) \mid v \in e\}| = k$.
Hence, proper $2$-colorings and polychromatic $2$-colorings are the same concept.
However, if $k \geq 3$, then every polychromatic $k$-coloring is also a proper $k$-coloring but the converse is not true in general.
In fact, for polychromatic colorings we always seek to \emph{maximize} the number of colors, as each polychromatic $k$-coloring, $k \geq 2$, also gives a polychromatic $(k-1)$-coloring by merging two color classes into one.

For polychromatic colorings of range capturing hypergraphs with respect to a given point set $V \subset \Rbb^d$ and range family $\Rcal$, we are particularly interested in the $m$-uniform\footnote{%
A hypergraph $H = (V,E)$ is \emph{$m$-uniform} if every hyperedge $e \in E$ has size $|e| = m$.
So $2$-uniform hypergraphs are just graphs (without loops), while $m$-uniform hypergraphs are sometimes also called $m$-graphs.%
} subhypergraph $\Hcal(V, \Rcal, m)$ that consists of all hyperedges in $\Hcal(V, \Rcal)$ of size exactly $m$.
Instead of fixing $m$ and then maximizing the $k$ for which polychromatic $k$-colorings of $\Hcal(V,\Rcal,m)$ exist, one usually consider the equivalent setup of fixing $k$ and minimizing $m$.

\begin{definition}
    For a range family $\Rcal$ and a positive integer $k$, the integer $m = m_\Rcal(k)$ is defined to be the smallest integer such that for every finite set of points $V \subset \Rbb^d$ there exists a polychromatic $k$-coloring of $\Hcal(V, \Rcal, m)$.
\end{definition}

Clearly, $m_\Rcal(k) \geq k$ since every hyperedge must contain $k$ different colors.
Moreover, we have $m_\Rcal(2) \leq m_\Rcal(3) \leq \cdots$.
But note that it is also possible that $m_\Rcal(k) = \infty$ for some $k$.
Namely this happens if for every positive integer $m$ there exists a finite set of points $V \subset \Rbb^d$ such that the corresponding hypergraph $\Hcal(V,\Rcal,m)$ has no polychromatic $k$-coloring.
In fact, throughout the over 40 years since their introduction by Pach~\cite{Pach80,Pach86}, we always observe the following surprising phenomenon for polychromatic $k$-colorings of geometric range spaces and the quantity $m_\Rcal(k)$ as a function of $k$:
Either we already have that $m_\Rcal(2) = \infty$, or the best known lower bounds are of the form $m_\Rcal(k) = \Omega(k)$ for all $k \geq 2$.

\begin{question}\label{oque:natural-family}
    Is there a geometric range family $\Rcal$ with $m_\Rcal(2) < \infty$ and $m_\Rcal(k) = \omega(k)$?
\end{question}

So, if the answer to~\cref{oque:natural-family} is 'No', then we always have either $m_\Rcal(2) = \infty$ or $m_\Rcal(k) = O(k)$.
With this paper, we seek to make progress on~\cref{oque:natural-family} by improving the upper bounds on $m_\Rcal(k)$ in several further cases from superlinear to $m_\Rcal(k) = O(k)$.
We do so by proving a stronger statement, namely the existence of so-called $t$-shallow hitting sets with $t = O(1)$; see~\cref{subsec:related-work,subsec:our-results} below for the formal definition and a detailed discussion.

\subsection{Related work}
\label{subsec:related-work}

There is a rich literature on numerous range families $\Rcal$, polychromatic colorings of their range capturing hypergraphs, and upper and lower bounds on $m_\Rcal(k)$ in terms of $k$~\cite{ACCIKLSST11,ACCCHHKLLMRU13,SY12,AKV17,chekan2022polychromatic,CKMU14,PP16,PPT13,CPST09,P13,PT10,Keszegh12,KLP16,KP19,Kovacs15,CKMPUV23,CKMU13,KP15}.
Let us mention just a few here, while the interested reader is invited to have a look at the slightly outdated survey article~\cite{PPT13} and the excellent website~\cite{Zoo} maintained by Keszegh and P\'alv\"olgyi.

\begin{enumerate}[label = (\arabic*)]
    \itemsep0pt
    \item[] \underline{(Some) known range families $\Rcal$ with $m_\Rcal(k) < \infty$ for all $k \geq 2$:}
    \item For axis-aligned strips $\Rst$ in $\Rbb^d$ it is known that $m_{\Rst}(k) = O_d(k \log k)$~\cite{ACCIKLSST11}.
    \label{1:known-strips}
    \item For bottomless rectangles $\Rbl$ in $\Rbb^2$ it is known that $1.67k \leq m_{\Rbl}(k) \leq 3k-2$~\cite{ACCCHHKLLMRU13}.
    \label{2:known-bottomless}
    \item For halfplanes $\Rcal$ in $\Rbb^2$ it is known that $m_\Rcal(k) = 2k-1$~\cite{SY12}.
    \label{3:known-halfplanes}
    \item For axis-aligned squares $\Rcal$ in $\Rbb^2$ it is known that $m_\Rcal(k) = O(k^{8.75})$~\cite{AKV17}.
    \label{4:known-squares}
    \item For bottomless and topless rectangles $\Rbl \cup \Rtl$ it is known that $m_\Rcal(k) = O(k^{8.75})$~\cite{chekan2022polychromatic}.
    \label{5:known-bottomless-topless}
    \item For translates of a convex polygon $\Rcal$ in $\Rbb^2$ it is known that $m_\Rcal(k) = O(k)$~\cite{GV09}.
    \label{6:known-polygon}
    \item For homothets of a triangle $\Rcal$ in $\Rbb^2$ it is known that $m_\Rcal(k) = O(k^{4.09})$~\cite{CKMU14,KP15}.
    \label{7:known-triangle}
    \item For translates of an octant $\Rcal$ in $\Rbb^3$ it is known that $m_\Rcal(k) = O(k^{5.09})$~\cite{CKMU14,KP15}.
    \label{8:known-octant}
    \item[] \underline{(Some) known range families $\Rcal$ with $m_\Rcal(k) = \infty$ for all $k \geq 2$:}
    \item For unit disks $\Rcal$ in $\Rbb^2$ it is known that $m_\Rcal(2) = \infty$~\cite{PP16}.
    \label{9:known-unit-disk}
    \item For strips $\Rcal$ in any direction in $\Rbb^2$ it is known that $m_\Rcal(2) = \infty$~\cite{PPT13}.
    \label{10:known-all-strips}
    \item For axis-aligned rectangles $\Rcal$ in $\Rbb^2$ it is known that $m_\Rcal(2) = \infty$~\cite{CPST09}.
    \label{11:known-rectangles}
    \item For bottomless rectangles and horizontal strips $\Rbl \cup \Rst^2$ we have $m_{\Rbl \cup \Rst^2}(2) = \infty$~\cite{chekan2022polychromatic}.\label{12:known-bottomless-and-strips}
\end{enumerate}
Crucially, let us mention again, that in each of \ref{1:known-strips}--\ref{8:known-octant} the best known lower bound on $m_\Rcal(k)$ is linear in $k$, and it might be (in the light of~\cref{oque:natural-family}) that in fact $m_\Rcal(k) = O(k)$ holds.

\smallskip

One particular tool to prove for a range family $\Rcal$ that $m_\Rcal(k) = O(k)$ are shallow hitting sets.
For a hypergraph $H = (V,E)$ and a positive integer $t$, a subset $X \subseteq V$ of vertices is a \emph{$t$-shallow hitting set} if 
\[
    1 \leq |e \cap X| \leq t \qquad \text{for every $e \in E$.}
\]
That is, $X$ contains at least one vertex of each hyperedge ($X$ is hitting) but at most $t$ vertices of each hyperedge ($X$ is $t$-shallow).
Shallow hitting sets for polychromatic colorings of range capturing hypergraphs have been used implicitly in~\cite{SY12}, while being developed as a general tool in~\cite{KP19,chekan2022polychromatic,CKMPUV23}.
Clearly, for the $m$-uniform hypergraph $\Hcal(V,\Rcal,m)$ taking $X = V$ would be an $m$-shallow hitting set.
But the challenge is to find $t$-shallow hitting sets with $t = O(1)$ being a constant\footnote{Recall that $m = m_\Rcal(k) \geq k$ is a growing function in $k$.} independent of $m$.
If we succeed, this implies $m_\Rcal(k) = O(k)$.

\begin{lemma}[Keszegh and P\'alv\"olgyi~\cite{KP19}]\label{lem:shallow-hitting-set}{\ \\}
    If for a shrinkable range family $\Rcal$ there exists a constant $t \geq 1$ such that for every $m \geq 1$ every hypergraph $\Hcal(V,\Rcal,m)$ admits a $t$-shallow hitting set, then $m_\Rcal(k) \leq t(k-1) + 1 = O(k)$.
\end{lemma}

Here, a range family $\Rcal$ is \emph{shrinkable} if for every finite set of points $V$, every positive integer $m$ and every hyperedge $e$ in $\Hcal(V,\Rcal,m)$ there exists a hyperedge $e'$ in $\Hcal(V,\Rcal,m-1)$ with $e' \subseteq e$.
Intuitively, we ``decrease the size'' of a range $R \in \Rcal$ with $R \cap V = e$ until the first point of $V$ drops out of the range.
In fact, all range families mentioned in this paper, except the translates of a convex polygon~\ref{6:known-polygon} and unit disks~\ref{9:known-unit-disk}, are shrinkable.

Smorodinsky and Yuditsky~\cite{SY12} prove that every $\Hcal(V,\Rcal,m)$ admits $2$-shallow hitting sets for $\Rcal$ being all halfplanes~\ref{3:known-halfplanes}, which implies $m_\Rcal(k) \leq 2k-1$ in this case.
This is extended to so-called ABA-free hypergraphs in~\cite{KP19} and unions of hypergraphs in~\cite{chekan2022polychromatic}.
On the other hand, for the family $\Rbl$ of all bottomless rectangles~\ref{2:known-bottomless} the bound $m_{\Rbl}(k) \leq 3k-2$ is not proven~\cite{ACCCHHKLLMRU13} by shallow hitting sets, and in fact Keszegh and P\'alv\"olgyi~\cite{KP19}, as well as Chekan and Ueckerdt~\cite{chekan2022polychromatic} ask whether these exist in this case.
For $\Rcal$ being all translates of a fixed convex polygon~\ref{6:known-polygon} the proof~\cite{GV09} for $m_\Rcal(k) = O(k)$ also involves shallow hitting sets, even though these are not explicitly stated as such, and this range family is not shrinkable anyways.
Finally, the family $\Rcal$ of all translates of an octant in $\Rbb^3$~\ref{8:known-octant} is the only case for which shallow hitting sets are known \emph{not} to exist~\cite{CKMPUV23}, which follows from a certain dual problem for bottomless rectangles.

\subsection{Our results}
\label{subsec:our-results}

We consider the range families mentioned at the beginning of~\cref{sec:introduction} of all axis-aligned strips $\Rst$ in $\Rbb^d$, all bottomless $\Rbl$ and all topless $\Rtl$ rectangles in $\Rbb^2$, as well as all unit-height rectangles $\Ruh$ in $\Rbb^2$.
We remark that for the axis-aligned strips $\Rst$ we could assume without loss of generality that these have unit-width.
In this sense, unit-height rectangles are a generalization of horizontal strips.
Additionally, unit-height rectangles are a generalization of bottomless and topless rectangles by ``choosing the unit very large''.
Thus, we can observe that $m_{\Ruh}(k) \geq m_{\Rst^2}(k)$ and $m_{\Ruh}(k) \geq m_{\Rbl \cup \Rtl}(k)$ hold for all $k$.

\begin{description}
    \itemsep 4pt
    \item[] \hspace{-0.5em}Our main results are the following:
    
    \item[\cref{sec:strips}.] The family $\Rst$ of all axis-aligned strips~\ref{1:known-strips} in $\Rbb^d$ allows for $t$-shallow hitting sets for some $t = t(d)=O(d)$, see~\cref{thm:union-of-strips}.
    This gives $m_{\Rst}(k) = O_d(k)$, improving the $O_d(k \log k)$-bound in~\cite{ACCIKLSST11}.
    We complement this with a lower bound construction giving $m_{\Rst}(k) \geq \Omega(k \log d)$, see~\cref{thm:strips-lower-bound-1}.
    This drastically improves the $m_{\Rst}(k) \geq 2 \cdot \left\lceil \frac{2d-1}{2d}\cdot k \right\rceil - 1$ lower bound in~\cite{ACCIKLSST11}.
    
    \item[\cref{sec:bottomless}.] The family $\Rbl$ of all bottomless rectangles~\ref{2:known-bottomless} in $\Rbb^2$ allows for $10$-shallow hitting sets, see~\cref{thm:bottomless-shallow-hitting-set}.
    This answers a question of Keszegh and P\'alv\"olgyi~\cite{KP19}, as well as Chekan and Ueckerdt~\cite{chekan2022polychromatic}, and provides a new proof that $m_{\Rbl}(k) = O(k)$.
    
    \item[\cref{sec:bottomless-topless}.] The family $\Rbl \cup \Rtl$ of all bottomless and topless rectangles~\ref{5:known-bottomless-topless} in $\Rbb^2$ allows for $21$-shallow hitting sets, see~\cref{thm:bottomless-topless-shallow-hitting-set}.
    This already proves that $m_{\Rbl\cup\Rtl}(k) = O(k)$, which we improve to $m_{\Rbl\cup\Rtl}(k) \leq 6k-3$, see \cref{thm:bottomless-topless-m(k)}.
    The family $\Ruh$ of all unit-height rectangles allows for $63$-shallow hitting sets, which already gives $m_{\Ruh}(k) = O(k)$ but can be improved to $m_{\Ruh}(k) \leq 12k - 7$, see~\cref{thm:unit-height-rectangles}.
\end{description}

\subparagraph{Notation.}
For a positive integer $n$ we sometimes use $[n] = \{1,\ldots,n\}$ for the set of the first $n$ integers.
Throughout this paper a hypergraph is a tuple $H=(V,E)$ consisting of a finite set $V$ of vertices (or points) and a finite multiset $E$ of hyperedges, each being a subset of $V$.
That is, hypergraphs may contain several distinct hyperedges forming the same subset of vertices.
Such hyperedges are sometimes called \emph{parallel}, \emph{multiedges}, or \emph{hyperedges of multiplicity $x$} for some $x \geq 2$.

\section{Polychromatic Colorings for Axis-Aligned Strips}
\label{sec:strips}

For a shorthand notation, let us define $m_d(k) = m_{\Rst}(k)$ for the range family $\Rst$ of all axis-aligned strips in $\Rbb^d$, $d \geq 2$.
As~\cite{ACCIKLSST11} pointed out, the problem of determining $m_d(k)$ for $\Rst$ can be seen purely combinatorial.
That is, the problem of determining $m_d(k)$ is equivalent to the following problem.
Given a finite set $V$ of size $n$ and $d$ bijections $\pi_1,\dots,\pi_d \colon \{1,\dots,n\} \to V$, we have to color the set $V$ in $k$ colors such that for each bijection $\pi_i$, every $m_d(k)$ consecutive elements contain an element of each color.
More formally, let $k$ and $d$ be positive integers.
Then, $m_d(k)$ is the least integer such that for any finite set $V$ of size $n$ and any $d$ bijections $\pi_1,\dots,\pi_d \colon \{1,\dots,n\} \to V$, there exists a coloring $c$ of $V$ with $k$ colors such that
\[
    \forall x \in [k] \; \forall i \in [d] \; \forall a \in [n-m_d(k)+1] \; \exists b \in [m_d(k)] \colon c(\pi_i(a+b-1)) = x.
\]

First, we list some known results for $m_d(k)$.
\begin{itemize}
    \item For $d=1$ it is obvious that $m_d(k) = m_1(k) = k$ for all $k$.
    \item For $d=2$ it holds that $m_d(k) = m_2(k) \leq 2k-1$ for all $k$~\cite{ACCIKLSST11}.
    \item For any $d \geq 2$ it holds that  $m_d(k) \leq k (4 \ln k + \ln d)$ for all $k$~\cite{ACCIKLSST11}.
    Thus, if $d$ is a constant, then $m_d(k) \leq O(k \log k)$.
    \item In~\cite{ACCIKLSST11} it is also proven that $m_d(k) \geq 2 \cdot \left\lceil \frac{2d-1}{2d}\cdot k \right\rceil - 1$, while in~\cite{PTT07} it is proven that for every $k$ we have $m_d(k) \to \infty$ as $d \to \infty$.
\end{itemize}

In this section, we prove that $m_d(k) \leq O(kd)$ for axis-aligned strips in $\Rbb^d$ with the method of shallow hitting sets.
This already improves the upper bound from~\cite{ACCIKLSST11} from $O_d(k \log k)$ to $O_d(k)$.
Using another result from the literature, we can even improve this further to $m_d(k) \leq O(k \log d)$.
We complement this by providing a lower bound of the form $m_d(k) \geq \Omega(k \log d)$.

\subsection{Upper Bounds}

Our upper bound uses a recent result about shallow hitting edge sets in regular uniform hypergraphs.
For a vertex $v \in V$ in a hypergraph $H = (V,E)$, the set of incident hyperedges at $v$ is denoted by $\incident(v) = \{e \in E \mid v \in e\}$.
Hypergraph $H$ is \emph{regular} if $|\incident(v)|$, the \emph{degree} of $v$, is the same for all vertices $v \in V$.
For an integer $t \geq 1$, a subset $M \subseteq E$ of hyperedges is a \emph{$t$-shallow hitting edge set} in $H = (V,E)$ if we have
\[
    1 \leq |M \cap \incident(v)| \leq t \qquad \text{for every $v \in V$}.
\]
That is, $1$-shallow hitting edge sets are exactly perfect matchings, while $t$-shallow hitting edge sets for $t \geq 2$ still cover each vertex at least once, but only at most $t$ times.
It turns out, that all regular $r$-uniform hypergraphs admit $t$-shallow hitting sets with $t$ only depending on the uniformity $r$, and not on the number of vertices or their degree.
Crucially, this result even holds for $r$-uniform hypergraphs with multiedges, i.e., where two or more hyperedges can correspond to the same set of $r$ vertices.

\begin{theorem}[Planken and Ueckerdt~\cite{PU23}]
    \label{thm:shallow-hitting-edge-set}
    Every $r$-uniform regular hypergraph $H$ (with possibly multiedges) has a $t(r)$-shallow hitting edge set with $t(r)=\euler r (1+o(1))$.
\end{theorem}

Here, $\euler = 2.71828\ldots$ denotes Euler's number.

Having~\cref{thm:shallow-hitting-edge-set}, we find shallow hitting sets for axis-aligned strips as follows.

\begin{theorem}
\label{thm:union-of-strips}
    Let $\Rst$ be the range family of all axis-aligned strips in $\Rbb^d$ and $m$ be a positive integer.
    Then, for every finite point set $V \subset \Rbb^d$, the hypergraph $\Hcal(V, \Rst, m)$ admits a $t(d)$-shallow hitting set, where $t(d) = 3 \euler d (1+o(1))$.
\end{theorem}
\begin{proof}
    Let $V$ be a finite set of points in $\Rbb^d$ of size $n$ and let $H = (V,E) = \Hcal(V,\Rst,m)$ be the corresponding $m$-uniform range capturing hypergraph induced by axis-aligned strips in $\Rbb^d$.
    We shall show that $H$ has a $t$-shallow hitting set, where $t = 3 \euler d (1+o(1))$.
    Let us set $r = \lfloor m/2 \rfloor$.
    We want to ensure that $n$ is a multiple of $r$.
    To this end, if $n = l \pmod r$ for some $l \neq 0$, then we add a set $A$ of $r-l$ new points, all of whose coordinates are larger than the coordinates in $V$.
    Observe that if $X'$ is a $t$-shallow hitting set in $\Hcal(V \cup A,\Rst,m)$, then $X = X' \cap V$ is a $t$-shallow hitting set in $H$, since every $e \in E$ is also a hyperedge in $\Hcal(V \cup A,\Rst,m)$.
    
    Thus, we may assume that $n = |V|$ and $r = \lfloor m/2 \rfloor$ divides $n$.
    For $i=1,\dots,d$, let $\pi_i \colon \{1,\dots, n\} \to V$ be the ordering of the points along the $i$-th coordinate axis.
    That is, $\pi_i(1) \in V$ is the point in $V$ with the lowest $i$-coordinate, $\pi_i(n) \in V$ is the point with the highest $i$-coordinate, and $\pi_i(1)_i < \cdots < \pi_i(n)_i$.
    Then, for each hyperedge $e$ in $\Hcal(V,\Rst^{i},m)$, the vertices in $e$ are $m$ consecutive elements in $\pi_i$.
    For $i=1,\dots,d$ and $j=0,\dots,n/r-1$, we define $W_{i,j}$ and $W_i$ to be
    \begin{align*}
        W_{i,j} &= \{ \pi_i(rj+1), \ldots, \pi_i( r(j+1) )\}
        \quad \text{and}\\
        W_i &= \{W_{i,j} \mid j=0,\dots,n/r-1\}~.
    \end{align*}
    In other words, each $W_i$ is a partition of the point set $V$ into $n/r$ parts of $r$ points with consecutive $i$-coordinates each.
    Thus, the hypergraph $H'=(V,E')$ with $E'=\bigcup_{i=1}^{d} W_i$ is $r$-uniform and $d$-regular.
    Let $H^*$ be the dual\footnote{%
    For a hypergraph $H = (V,E)$ its \emph{dual} is the hypergraph $H^*=(V^*,E^*)$ with vertex-set $V^*=E$ and edge-set $E^* = \{\incident(v) \mid v \in V\}$.
    Note that $H^*$ may have parallel hyperedges.%
    } hypergraph of $H'$.
    Then, $H^*$ is $d$-uniform and $r$-regular, with the hyperedges of $H^*$ corresponding to the vertices of $H'$, hence the points in $V$.
    By~\cref{thm:shallow-hitting-edge-set}, $H^*$ has a $t'$-shallow hitting edge set, where $t'=t'(d)=\euler d (1+o(1))$.
    Then, the corresponding set of vertices $X$ of $H'$ is a $t'$-shallow hitting set in $H'$.
    With $t = 3t'$, all that remains to show is that $X$ is a $3t'$-shallow hitting set in $H = \Hcal(V,\Rst,m)$.

    Let $e$ be any hyperedge in $H$.
    Since $|e|=m$, and since every hyperedge in $H'$ has size $r=\lfloor m/2 \rfloor$, there exists a hyperedge $e'$ in $H'$ with $e' \subseteq e$.
    Since $X$ is hitting in $H'$, it is also hitting in $H$.
    Moreover, for every hyperedge $e$ in $H$ we can find three hyperedges $e'_1,e'_2,e'_3$ in $H'$ with $e \subseteq e'_1 \cup e'_2 \cup e'_3$.
    Thus, since $X$ is $t'$-shallow in $H'$, it is $3t'$-shallow in $H$.
\end{proof}

Combining \cref{thm:union-of-strips} and \cref{lem:shallow-hitting-set}, we already obtain $m_d(k) \leq O(kd) \leq O_d(k)$, which beats the previous $O_d(k \log k)$ bound from~\cite{ACCIKLSST11}.
\begin{corollary}
    \label{cor:union-of-strips}
    For the range family $\Rst$ of all axis-aligned strips in $\Rbb^d$ and every integer $k \geq 2$ we have $m_{\Rst}(k) = m_d(k) \leq 3 \euler d (1+f(d)) \cdot k$ where $f(d) \to 0$ as $d \to \infty$.
\end{corollary}

Using the following result of Bollob{\'a}s et al.~\cite{Bollob_s_2013}, we can even improve to $m_d(k) = O(k \log d)$.

\begin{theorem}[Bollob{\'{a}}s, Pritchard, Rothvoss and Scott~\cite{Bollob_s_2013}]
\label{thm:bollobas-polychromatic-edge-coloring}
    Every $r$-uniform $\Delta$-regular hypergraph (with possibly multiedges) has a polychromatic $k$-edge-coloring\footnote{a coloring of the hyperedges such that each vertex is incident to a hyperedge of every color} with $k \geq \Delta/(\ln r + O(\ln \ln r))$.
\end{theorem}

\begin{corollary}
    \label{cor:union-of-strips-2}
    For the range family $\Rst$ of all axis-aligned strips in $\Rbb^d$ and every integer $k \geq 2$ we have $m_{\Rst}(k) = m_d(k) \leq 2 k (\ln d + O(\ln \ln d))$.
\end{corollary}
\begin{proof}
    Let $V$ be a finite set of points in $\Rbb^d$ of size $n$.
    Let $r = \lceil k (\ln d + O(\ln \ln d)) \rceil $ and $m = 2r$.
    We show that the $m$-uniform range capturing hypergraph $H=\Hcal(V,\Rst,m)$ induced by axis-aligned strips in $\Rbb^d$ admits a polychromatic $k$-coloring.
    
    First, we construct the $r$-uniform $d$-regular hypergraph $H'$ as in the proof of \cref{thm:union-of-strips} and define its dual hypergraph to be $H^*$ (which is $d$-uniform and $r$-regular).
    By \cref{thm:bollobas-polychromatic-edge-coloring}, $H^*$ admits a polychromatic $k'$-edge-coloring with $k' \geq r/(\ln d + O(\ln \ln d)) \geq k$, i.e., every vertex of $H^*$ is incident to an edge of every color.
    Therefore, its dual $H'$ admits a polychromatic $k$-coloring $\psi$.

    It remains to show that $\psi$ is a polychromatic $k$-coloring of $H$.
    Let $e$ be any hyperedge in $H$.
    Since $|e|=m$ and since every hyperedge in $H'$ has size $r=m/2$, there exists a hyperedge $e'$ in $H'$ with $e' \subseteq e$.
    Since $e'$ is colored polychromatically, so is $e$.
\end{proof}

\subsection{Lower Bounds}

We seek to give a lower bound on $m_d(k) = m_{\Rst}(k)$ for the range family $\Rst$ of all axis-aligned strips in $\Rbb^d$.
That is, for every $d,k \geq 1$ we construct a point set $V = V_{d,k}$ in $\Rbb^d$ such that for some (hopefully large) $m$ the range capturing hypergraph $\Hcal(V,\Rst,m)$ admits no polychromatic $k$-coloring.
Then it follows that $m_d(k) \geq m+1$.

As a first step towards the desired point sets, we first present a construction of $r$-uniform $r$-partite\footnote{%
A hypergraph $H = (V,E)$ is \emph{$r$-partite} if there exists a partition $V = V_1 \dot\cup \cdots \dot\cup V_r$ such that for every $e \in E$ and every $i \in [r]$ we have $|e \cap V_i| \leq 1$.
The sets $V_1,\ldots,V_r$ are then called that \emph{parts} of $H$.%
} $t$-regular hypergraphs with $t$ being relatively large in terms of $r$, which admit no $(t-1)$-shallow hitting edge sets.

\begin{theorem} 
    \label{thm:simple-construction}
    Let $t \geq 2$ be an integer.
    There exists an $r$-uniform $r$-partite $t$-regular hypergraph with parts of size two that has no $(t-1)$-shallow hitting edge set, where $r = \binom{2t}{t}/2 \leq 4^t$, i.e., $t \geq \log_4(r)$.
\end{theorem}
\begin{proof}
    Let $H=(V,E)$ be the hypergraph with $V=\{1,\dots,2t\}$ and $E=\binom{V}{t}$, i.e., the hyperedges are all $t$-element subsets of $V$.
    Observe that $H$ is $t$-uniform, $r$-regular with $r=\binom{2t}{t}/2$.
    Moreover $H$ is the union of $r$ perfect matchings, each of the form $A,B \in \binom{V}{t}$ with $B = V - A$.
    
    First, we show that $H$ has no $(t-1)$-shallow hitting (vertex) set.
    To this end let $X \subseteq V$ be any set of vertices in $H$.
    If $|X| \leq t$, then $|V - X| \geq t$ and there exists a hyperedge $e \subseteq V - X$ which is not covered by $X$.
    In this case, $X$ is not hitting.
    If $|X| \geq t$, then there exists a hyperedge $e \subseteq X$.
    Since $e$ has size $t$, the set $X$ is not $(t-1)$-shallow.
    
    Now consider the dual hypergraph $H^*$ of $H$.
    Then, $H^*$ is an $r$-uniform $r$-partite $t$-regular hypergraph.
    Two vertices $v$ and $v'$ in $H^*$ (recall that $v,v'$ are $t$-subsets of $\{1,\dots,2t\}$) are in the same part if and only if $v'=\{1,\dots,2t\} - v$.
    Since $H$ has no $(t-1)$-shallow hitting (vertex) set, $H^*$ has no $(t-1)$-shallow hitting edge set.
\end{proof}

In the next theorem, we seek to find lower bounds for $m_d(k)$ for axis-aligned strips in $\Rbb^d$. 
For that, we use the constructions in \cref{thm:simple-construction}.
We reduce the problem of finding lower bounds for $m_d(k)$ to the problem of finding lower bounds of $m'_d(k)$, which is defined in the following.
We define $m'_d(k)$ to be the least integer $m'$ such that every $d$-uniform $d$-partite $m'$-regular hypergraph admits a \emph{polychromatic edge-coloring} with $k$ colors, that is, a coloring of the hyperedges such that each vertex is incident to a hyperedge of every color.
Note that in a $d$-uniform $d$-partite hypergraph every hyperedge uses exactly one vertex in each part.
If such a hypergraph is additionally regular, it follows that each part has the same size.

In order to obtain a lower bound on $m_d(k)$, we first show that $m_d(k) \geq m'_d(k)$, and afterwards prove lower bounds for $m'_d(k)$.

\begin{lemma}\label{lem:reduction-of-md}
    For every $d$ and $k$ we have $m_d(k) \geq m'_d(k)$.
\end{lemma}
\begin{proof}
    Let $m = m_d(k)$.
    Then every range capturing hypergraph $H=\Hcal(V,\Rst,m)$ (with $V \subset \Rbb^d$ finite) admits a polychromatic $k$-coloring of its vertices.
    Let $H'=(V',E')$ be any $d$-uniform $d$-partite $m$-regular hypergraph with parts $V'_1,\dots,V'_d$ of size $n$ and $V'_i=\{v_{i,1},\dots,v_{i,n}\}$ for $i=1,\dots,d$.
    We deduce from $H'$ the following finite point set $V \subset \Rbb^d$, which defines the range capturing hypergraph $H=\Hcal(V, \Rst, m)$.
    For each part $V'_i$ of $H'$, let $\pi_i \colon E \to \{1,\dots,nm\}$ be a bijection that satisfies the following condition.
    For two hyperedges $e$ and $e'$ with $e \cap V'_i = \{v_{i,j}\}$ and $e' \cap V_i = \{v_{i,j'}\}$ and $j < j'$ it holds that $\pi_i(e) < \pi_i(e')$.
    Now, let the point set be $V=\{ (\pi_1(e),\dots,\pi_d(e)) \mid e \in E\} \subset \Rbb^d$.
    Note that, for every vertex $v_{i,j}$ in $V'_i$, its incident hyperedges $\incident(v_{i,j}) \subseteq E'$ correspond to points in $V$ that are consecutive in the $i$-th dimension.
    Recall that $H$ admits a polychromatic $k$-coloring of its vertices, i.e., each $m$-set of points that are consecutive in some dimension $i$ contains points of all $k$ colors.
    Then it follows that $H'$ admits a polychromatic $k$-coloring of its hyperedges.
\end{proof}

Having \cref{lem:reduction-of-md}, it remains to prove a lower bound on $m'_d(k)$.

\begin{theorem}
\label{thm:strips-lower-bound-1}
    $m_d(k) \geq m'_d(k) > \frac{1}{2} \left( \log_2 d - 1 \right) \cdot \lfloor k/2 \rfloor$.
\end{theorem}
\begin{proof}
    Let $k$ and $d$ be positive integers.
    Let $t$ be the largest integer such that $\binom{2t}{t} / 2 \leq d$.
    Let $d_0 = \binom{2t}{t} / 2 \leq 4^t / 2$ and observe that $d_0 \leq d \leq 4 d_0$.
    Let $H_0$ be the $d_0$-uniform $d_0$-partite $t$-regular hypergraph with two vertices per part from \cref{thm:simple-construction}.
    Observe that if $M$ is any subset of hyperedges in $H_0$ that together contain all vertices of $H_0$, called a \emph{hitting edge set}, then $M$ has size at least $t+1$.
    
    We construct the hypergraph $H$ by replacing each hyperedge of $H_0$ by a multiedge of multiplicity $\lfloor k/2 \rfloor$.
    Then, $H$ is a $d_0$-uniform $d_0$-partite $(t \lfloor k/2 \rfloor)$-regular hypergraph and each hitting edge set of $H$ has size at least $t+1$.
    Observe that $|E(H)| = 2 t \lfloor k/2 \rfloor \leq t k$, since each part of $H$ has size~$2$.
    
    Now assume for a contradiction that $H$ admits a polychromatic $k$-coloring of its hyperedges, i.e., a $k$-coloring of the hyperedges of $H$ such that every vertex is incident to at least one hyperedge of each color.
    Since each color class is a hitting edge set, each color class contains at least $t+1$ hyperedges.
    Thus, the number of hyperedges is $|E(H)| \geq (t+1) k > t k$, a contradiction.

    With $d_0 \leq d \leq 4 d_0$, we conclude:
    \begin{multline*}
         m'_d(k) \geq m'_{d_0}(k) 
        > t \left\lfloor \frac{k}{2} \right\rfloor
        \geq \frac{1}{2} \log_2(2 d_0) \left\lfloor \frac{k}{2} \right\rfloor
        \geq \frac{1}{2} \log_2(d / 2) \left\lfloor \frac{k}{2} \right\rfloor
        = \frac{1}{2} \left(\log_2 d - 1 \right) \left\lfloor \frac{k}{2} \right\rfloor
    \end{multline*}
\end{proof}

\begin{remark}
    In~\cite{PU23} there is a more sophisticated (compared to \cref{thm:simple-construction}) construction of $r$-uniform $r$-partite regular hypergraphs with two vertices per part that have no $(t-1)$-shallow hitting edge set with a slightly better bound for $t$, namely with $t = \log_2(r+1)$.
    Using this construction instead, an analogous proof as in \cref{thm:strips-lower-bound-1} then gives that $m_d(k) \geq m'_d(k) > \frac{1}{2}\left(\left(\log_2 d - 1\right) \cdot k - d\right)$, which is better by a factor of~$2$ as long as $k > d$.
\end{remark}

\section{Bottomless Rectangles}
\label{sec:bottomless}

For the range family $\Rbl$ of all bottomless rectangles in $\Rbb^2$, it is known that $m_{\Rbl}(k) = O(k)$~\cite{ACCCHHKLLMRU13}.

\begin{theorem}[Asinowski et al.~\cite{ACCCHHKLLMRU13}]{\ \\}
    \label{thm:bottomless-m(k)}
    For the range family $\Rbl$ of all bottomless rectangles in $\Rbb^2$ we have $m_{\Rbl}(k) \leq 3k-2$.
\end{theorem}

However, the proof in~\cite{ACCCHHKLLMRU13} does not go via shallow hitting sets, and it is also not clear how to adjust it to give shallow hitting sets.
In fact, Keszegh and P\'alv\"olgyi~\cite{KP19} ask whether there exists a constant $t$ such that for every $V$ the hypergraph $\Hcal(V,\Rbl,m)$ admits a $t$-shallow hitting set.
We answer this question in the positive.

\begin{theorem}
\label{thm:bottomless-shallow-hitting-set}
    Let $\Rbl$ be the range family of all bottomless rectangles in $\Rbb^2$ and $m$ be a positive integer.
    Then for any finite point set $V \subset \Rbb^2$ the hypergraph $\mathcal{H}(V,\Rbl,m)$ admits a $10$-shallow hitting set $X \subseteq V$.
\end{theorem}
\begin{proof}
    Let $V \subset \Rbb^2$ be any finite point set and let $V =\{p_1, \dots, p_n\}$ with $y(p_1) < \dots < y(p_n)$.
    (Recall that $y(p)$ denotes the $y$-coordinate of a point $p \in \Rbb^2$.)
    Let $w = \lfloor (m+3) / 4 \rfloor$ and note that $4w-3 \leq m \leq 4w$.
    We can assume that $m > 10$ since for $m \leq 10$, the point set $X = V$ is a $10$-shallow hitting set of $\Hcal(V,\Rbl,m)$.
    Moreover, we can assume that $|V| \geq m \geq 4w-3$ since otherwise $\mathcal{H}(V,\Rbl,m)$ has no hyperedges.
    
    We shall perform a sweep-line algorithm that goes through the points in order of increasing $y$-coordinates and builds the desired $10$-shallow hitting set by selecting one by one points to be included in $X$, without ever revoking such decision.
    Such an algorithm is called \emph{semi-online} as its choices will be independent of the points above the current sweep-line (with larger $y$-coordinates).
    During the sweep we consider the $x$-coordinates of the points below the sweep-line.
    Note that if $m$ points have consecutive $x$-coordinates among those below the sweep-line, then these $m$ points form a hyperedge in $\Hcal(V,\Rbl,m)$, as verified by a bottomless rectangle whose top side lies on the sweep-line.
    And conversely, if some $m$ points of $V$ form a hyperedge in $\Hcal(V,\Rbl,m)$, then these have consecutive $x$-coordinates among those below the sweep-line at the time that the sweep-line contains the top side of a corresponding bottomless rectangle.
    
    We start the sweep-line algorithm with step $j = w$.
    In step $j$, $j \geq w$, we consider the points $V_j = \{p_1,\dots, p_j\}$, i.e., the $j$ points with the lowest $y$-coordinates.
    We construct a set $X_j \subseteq V_j$ of \emph{black points} (points that are definitely in the final set $X$) and a set of \emph{white points} $W_j \subseteq V_j$ (points that are definitely \emph{not} in the final set $X$) such that (for $j > w$) we have $X_{j-1} \subseteq X_j$ and $W_{j-1} \subseteq W_j$.
    We refer to the points that are neither white nor black as \emph{uncolored points}.
    Additionally, we maintain a partition $\Rbb = A_{j,1} \dot\cup \cdots \dot\cup A_{j,l}$ of the real line $\Rbb$ into $l$, for some $l$, pairwise disjoint intervals $A_{j,i}$ with $A_{j,1} = (-\infty,a_1), A_{j,2}=[a_1,a_2), \dots, A_{j,l} = [a_{l-1}, \infty)$ with $-\infty < a_1 < a_2 < \dots < a_{l-1} < \infty$.
    We define $V_{j,i}$ to be the set of points $p \in V_j$ with $x$-coordinate $x(p) \in A_{j,i}$. 
    While executing the sweep-line algorithm, we maintain the following invariants.
    \begin{itemize}
        \item Each $V_{j,i}$ contains exactly one black point and $w-1$ white points, i.e., $|V_{j,i} \cap X_j| = 1$ and $|V_{j,i} \cap W_j| = w-1$.
        \item Each $V_{j,i}$ has size $w \leq |V_{j,i}| \leq 2w-1$.
    \end{itemize}
    
    We start with step $j=w$ as follows.
    The set of black points is $X_w = \{p_1\}$, the set of white points is $W_w = \{p_2, \dots, p_w\}$ and $\Rbb = (-\infty, \infty)$ is the partition of $\Rbb$ into one set.
    Clearly, all conditions are satisfied.
    
    Now, suppose that $X_j$, $W_j$, and the partition $\Rbb = A_{j,1} \dot\cup \cdots \dot\cup A_{j,l}$ are given as the result of step $j$.
    In the next step $j+1$, we consider the set $V_{j+1} = V_j \cup \{p_{j+1}\}$.
    Let $A_{j,i'}$ be the interval with $x(p_{j+1}) \in A_{j,i'}$.
    We distinguish two cases.
    If $|V_{j,i}| < 2w-1$, then we set $X_{j+1}=X_j$, $W_{j+1} = W_j$ and $A_{j+1,i} = A_{j,i}$ for all $i=1,\dots,l$ for the next step $j+1$.
    Then, $|V_{j+1,i'}| \leq 2w-1$ and all conditions are again satisfied.
    Otherwise, assume that $|V_{j,i}| = 2w-1$.
    Let $q_1,\dots,q_{2w}$ be the points in $V_{j,i} \cup \{p_{j+1}\}$ ordered by their $x$-coordinate, i.e., $x(q_1)<\cdots<x(q_{2w})$, and define $a' = x(q_{w+1})$.
    Then, we define the partition 
    \begin{align*}
        \Rbb &= A_{j+1,1} \,\dot\cup\, \cdots \,\dot\cup\, A_{j+1,l+1}\\
            &= (-\infty, a_1) \,\dot\cup\, \cdots \,\dot\cup\, [a_{i'-1}, a') \,\dot\cup\, [a', a_{i'}) \,\dot\cup\, \cdots \,\dot\cup\, [a_{l-1}, \infty)~.
    \end{align*}
    That is, we split the interval $A_{j,i'}=[a_{i'-1},a_{i'})$ from step $j$ into two intervals $[a_{i'-1}, a')$ and $[a', a_{i'})$.
    Observe that $|V_{j+1,i'}| = w = |V_{j+1,i'+1}|$.
    Since there is exactly one black point in $V_{j,i'}$ (i.e., $|X_j \cap V_{j,i'}|=1$), there is exactly one black point of $X_j$ in $V_{j+1,i'} \cup V_{j+1,i'+1}$.
    By symmetry, assume that this black point is contained in $V_{j+1,i'}$ and therefore, $V_{j+1,i'+1}$ has no black point in $X_j$.
    Now we color all uncolored points in $V_{j+1,i'}$ white.
    Then, $V_{j+1,i'}$ contains exactly one black and $w-1$ white points.
    Since $V_{j,i'}$ has at most $w-1$ white points and $V_{j+1,i'+1} \subseteq V_{j,i'} \cup \{p_{j+1}\}$, the set $V_{j+1,i'+1}$ has at most $w-1$ white points of $W_j$, too.
    Thus, there exists an uncolored point $q$ in $V_{j+1,i'+1}$.
    We color $q$ black and all other uncolored points in $V_{j+1,i'+1}$ white.
    Then, $V_{j+1,i'+1}$ contains exactly one black and $w-1$ white points.
    This completes step $j+1$.
    Note that both invariants are again satisfied.

    After step $n = |V|$, we have considered all points in $V$.
    Let $X = X_n$ be the set of black points as result of the last step.
    We show that $X$ is a $10$-shallow hitting set of $\Hcal(V,\Rbl,m)$.
    
    \begin{claim}
        \label{claim:bottomless-hitting}
        $X$ is hitting in $\Hcal(V,\Rbl,m)$.
    \end{claim}
    \begin{claimproof}
        Let $R=[a,b] \times (-\infty,c]$ be a bottomless rectangle that contains $m$ points of $V$, i.e., $|R \cap V| = m$.
        Let $p$ be the topmost point in $R \cap V$ and consider the state of the sweep-line algorithm right after $p$ is inserted, that is, $p=p_j$ for some $j$ and step $j$ is finished.
        Then, the points in $V_j$ with $x$-coordinate in the interval $[a,b]$ are exactly the points in $R \cap V$.
        Since we have $|R \cap V| = m \geq 4w - 3$ and each $V_{j,i}$ has size at most $2w-1$, there exists a $V_{j,i'}$ with $V_{j,i'} \subseteq R \cap V$.
        Since $V_{j,i'}$ contains a black point ($V_{j,i'} \cap X \neq \emptyset$), $R \cap V$ contains a black point too ($R \cap X \neq \emptyset$) and $X$ is hitting.
    \end{claimproof}
    
    \begin{claim}
        \label{claim:shallowness-of-intervals}
        $|X \cap V_{j,i}| \leq 2$ for every $V_{j,i}$.
    \end{claim}
    \begin{claimproof}
        By the invariants above it holds that $|X_j \cap V_{j,i}| = 1$ (Be aware of the difference between $X \cap V_{j,i}$ and $X_j \cap V_{j,i}$.) and $|W_j \cap V_{j,i}| = w-1$.
        Moreover, observe that whenever an uncolored point $p \in V_{j,i}$ is colored black, all other uncolored points in $V_{j,i}$ are colored white.
        Since white points are definitely not contained in $X$, the claim follows.
    \end{claimproof}
    
    \begin{claim}
        \label{claim:bottomless-shallow}
        $X$ is $10$-shallow in $\Hcal(V,\Rbl,m)$.
    \end{claim}
    \begin{claimproof}
        Again, let $R$ be a bottomless rectangle that contains $m$ points of $V$, i.e., $|R \cap V| = m$, and let $p$ be the topmost point in $R \cap V$.
        Consider the state of the sweep-line algorithm after $p$ is inserted, that is, $p=p_j$ for some $j$ and step $j$ is finished.
        We have $|R \cap V| = m \leq 4w$ and each $V_{j,i}$ contains at least $w$ points.
        Therefore, there exist at most five sets $V_{j,i}$ with $V_{j,i} \cap R \neq \emptyset$.
        By \cref{claim:shallowness-of-intervals}, each $V_{j,i}$ contains at most two points of $X$.
        Therefore, $|R \cap X| \leq 5 \cdot 2 = 10$ and $X$ is $10$-shallow.
    \end{claimproof}
    
    By \cref{claim:bottomless-hitting,claim:bottomless-shallow}, $X$ is a $10$-shallow hitting set in $\Hcal(V,\Rbl,m)$.
\end{proof}

\begin{remark}
    The procedure in the proof of \cref{thm:bottomless-shallow-hitting-set} can be modified to directly get a polychromatic coloring of the range capturing hypergraph induced by bottomless rectangles.
    Let $w=\lfloor (m+3)/4 \rfloor$ be as in the proof of \cref{thm:bottomless-shallow-hitting-set}.
    Instead of carrying black and white sets, we carry a partial $w$-coloring (i.e., a $w$-coloring of some vertices on the sweep-line) such that in each step $j \geq 1$, every set $V_{j,i}$ of points contains every color exactly once.
    At the end of the algorithm, we get a partial $w$-coloring of all vertices.
    We complete this to a $w$-coloring by assigning colors to the uncolored vertices such that every $V_{n,i}$ contains every color at most twice.
    Note that every color class is a 10-shallow hitting set in $\Hcal(V,\Rbl,m)$.
    By setting $w=k$, one can observe that this $k$-coloring is polychromatic in $\Hcal(V,\Rbl,m)$, which gives a proof of $m_{\Rbl}(k) \leq 4k-3$.
    Moreover, if $e$ is an edge in $\Hcal(V,\Rbl)$, not necessarily of size $m$, and $n_1,n_2$ denote the size of two color classes in $e$ then it holds that $n_1 \leq 4+2 n_2 \leq 4(n_2+1)$.
    Therefore, this $k$-coloring is 4-balanced\footnote{We define and discuss $t$-balanced colorings in \cref{sec:conclusions}.} in $\Hcal(V,\Rbl)$.
\end{remark}

\section{Bottomless and Topless Rectangles}
\label{sec:bottomless-topless}

Chekan and Ueckerdt~\cite{chekan2022polychromatic} showed that $m_{\Rbl\cup\Rtl}(k) \leq O(k^{8.75})$ for the range family $\Rbl\cup\Rtl$ of bottomless and topless rectangles by a reduction to the family $\Rcal$ of all axis-aligned squares, and using that $m_\Rcal(k) = O(k^{8.75})$ in this case~\cite{AKV17}.
We improve the upper bound on $m(k)$ for the case $\Rbl\cup\Rtl$ to $O(k)$ in the following theorem, by a simple reduction to the case $\Rbl$ of just all bottomless rectangles, and the case $\Rtl$ of just all topless rectangles.
Observe that we clearly have $m_{\Rbl}(k) = m_{\Rtl}(k)$ for all $k$, and recall that $m_{\Rbl}(k) \leq 3k-2$ according to~\cite{ACCCHHKLLMRU13} (see \cref{thm:bottomless-m(k)}).

\begin{theorem}
    \label{thm:bottomless-topless-m(k)}
    For $\Rbl \cup \Rtl$ the range family of all bottomless and topless rectangles in $\Rbb^2$, we have $m_{\Rbl\cup\Rtl}(k) \leq 2 m_{\Rbl}(k) + 1 \leq 6k-3$.
\end{theorem}
\begin{proof}
    Let $m' = m_{\Rbl}(k) = m_{\Rtl}(k)$ and $m = 2 m' + 1$.
    Let $V = \{p_1,\dots,p_n\} \subset \Rbb^2$ be a finite point set with $x(p_1) < \cdots < x(p_n)$.
    (Recall that $x(p)$ denotes the $x$-coordinate of a point $p \in \Rbb^2$.)
    We partition the set $V$ into two sets $A$ and $B$.
    For each pair $\{p_{2i-1}, p_{2i}\}$, we put the vertex with the lower $y$-coordinate into set $A$ and the point with the larger $y$-coordinate into set $B$, see \cref{fig:topless-bottomless-m(k)}(a).

    \begin{figure}[htb]
        \centering
        \includegraphics{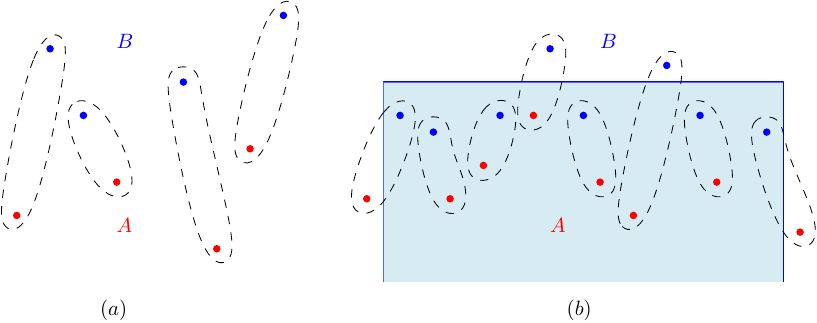}
        \caption{
            (a) For each pair $\{p_{2i-1}, p_{2i}\}$, the vertex with lower $y$-coordinate is in the set $A$ and the other vertex is in $B$. (b) If a bottomless rectangle contains $m$ points, then it contains at least $\lceil (m-2)/2 \rceil$ points of $A$.
        }
        \label{fig:topless-bottomless-m(k)}
    \end{figure}
    
    Consider a polychromatic $k$-coloring $c_1 \colon A \to \{1,\ldots,k\}$ of the hypergraph $\Hcal(A, \Rbl, m')$ and a polychromatic $k$-coloring $c_2 \colon B \to \{1,\ldots,k\}$ of the hypergraph $\Hcal(B, \Rtl, m')$.
    As $V = A \dot\cup B$, this naturally defines a $k$-coloring $c \colon V \to \{1,\ldots,k\}$ of $\Hcal(V, \Rbl \cup \Rtl, m)$.
    To see that coloring $c$ is polychromatic, let $e$ be a hyperedge in $\Hcal(V, \Rbl \cup \Rtl, m)$ induced by a bottomless or topless rectangle $R \in \Rbl \cup \Rtl$.
    If $R \in \Rbl$, then $R$ contains at least $\lceil (m - 2)/2 \rceil = \lceil (2m'-1)/2 \rceil = m'$ points from $A$, see \cref{fig:topless-bottomless-m(k)}(b).
    Thus, $e \cap A$ is colored polychromatically in $\Hcal(A, \Rbl, m')$ and hence $e$ is colored polychromatically in $\Hcal(V, \Rbl \cup \Rtl, m)$.
    Symmetrically, if $R \in \Rtl$, then $R$ contains at least $m'$ points from $B$, thus $R \cap B$ contains all $k$ colors under $c_2$, and thus $e = R \cap V \supseteq R \cap B$ contains all $k$ colors under $c$.
\end{proof}

According to \cref{thm:bottomless-topless-m(k)} we have $m_{\Rbl\cup \Rtl}(k) = O(k)$.
However, the proof relies on the polychromatic coloring from~\cite{ACCCHHKLLMRU13} and thus does not give shallow hitting sets, which (up to the constants) is the stronger statement.
In fact, even if we had a shallow hitting set $X$ for $\Hcal(A, \Rbl, m')$ and a shallow hitting set $Y$ for $\Hcal(B, \Rtl, m')$ ($A$ and $B$ as in the proof above), their union $X \cup Y$ would be hitting, but not necessarily shallow.

Recall that a subset $X$ of the vertices of a hypergraph $H = (V,E)$ is \emph{hitting} if $|X \cap e| \geq 1$ for every $e \in E$, and \emph{$t$-shallow} if $|X \cap e| \leq t$ for every $e \in E$.
In order to prove the existence of shallow hitting sets for $\Rbl \cup \Rtl$, we shall first find a shallow hitting set for $\Rbl$, which is also shallow (but not necessarily hitting) for $\Rtl$.
A similar approach has been done in~\cite{chekan2022polychromatic}.

\begin{lemma}
\label{lem:bottomless-topless-lemma}
    Let $V \subset \Rbb^2$ be a finite point set and $m$ be a positive integer.
    Then, there exists a set $X \subseteq V$ such that
    \begin{itemize}
        \itemsep0pt
        \item $X$ is a $14$-shallow hitting set of $\Hcal(V,\Rbl,m)$ and
        \item $X$ is a $7$-shallow set of $\Hcal(V,\Rtl,m)$.
    \end{itemize}
\end{lemma}
\begin{proof}
    Let $V=\{p_1, \dots, p_n\} \subset \Rbb^2$ be the finite point set with $y(p_1) < \cdots < y(p_n)$.
    Let $w = \lfloor (m-1)/6 \rfloor$ and observe that $6w + 1 \leq m \leq 6w+6$.
    We can assume that $m > 7$ since for $m \leq 7$, the point set $X = V$ is a $14$-shallow hitting set of $\Hcal(V,\Rbl,m)$ and $7$-shallow set of $\Hcal(V, \Rtl, m)$.
    Moreover, we can assume that $|V| \geq m \geq 6w+1$ since otherwise the range capturing hypergraphs contain no hyperedges.

    We perform a sweep-line algorithm to obtain the desired set $X$, similarly to the proof of \cref{thm:bottomless-shallow-hitting-set}.
    Again, we select one by one points below (or on) the current sweep-line to be included in $X$ and mark others to be not included, without ever revoking a decision.
    However, the algorithm will not be semi-online as its decisions shall also depend on points in $V$ above the current sweep-line.
    
    We start with step~$j=w$.
    In step $j$, $j \geq w$, we consider the points $V_j = \{p_1,\dots,p_j\}$ and construct a set $X_j \subseteq V_j$ of \emph{black points} (points that are definitely in the final set $X$) and a set $W_j \subseteq V_j$ of \emph{white points} (points that are definitely \emph{not} in the final set $X$) such that (for $j > w$) we have $X_{j-1} \subseteq X_j$ and $W_{j-1} \subseteq W_j$ for all $j$.
    We refer to the points that are neither white nor black as \emph{uncolored points}.
    Moreover, we maintain a partition $\Rbb = A_{j,1} \dot\cup \cdots \dot\cup A_{j,l}$ of the real line $\Rbb$ into $l$ pairwise disjoint intervals $A_{j,i}$, $i=1,\ldots,l$, with $A_{j,1}=(-\infty, a_1)$, $A_{j,2}=[a_1, a_2)$, $\dots$, $A_{j,l}=[a_{l-1}, \infty)$ with $-\infty < a_1 < a_2 < \cdots < a_{l-1} < \infty$.
    
    We define $V_{j,i}$ to be the set of points $p \in V_j$ with $x$-coordinate $x(p) \in A_{j,i}$ and $U_{j,i}$ to be the set of uncolored points in $V$ (i.e., with arbitrary $y$-coordinate) with $x$-coordinate $x(p) \in A_{j,i}$.
    Moreover, let us define $X_{j,i} = X_j \cap V_{j,i}$ and $W_{j,i} = W_j \cap V_{j,i}$ as the set of black and white points with $x$-coordinates in the interval $A_{j,i}$, respectively.
    An interval $A_{j,i}$ is said to be \emph{alive} if $|U_{j,i}| \geq 2w+2$, and \emph{dead} otherwise.
    We refer to the corresponding $V_{j,i}$ and $U_{j,i}$ as alive respectively dead.
    Additionally, we maintain for some intervals $A_{j,i}$ a \emph{premarked split point $s_{j,i} \in A_{j,i}$}, and for such intervals define $A_{j,i}^0 = [a_{i-1}, s_{j,i})$ and $A_{j,i}^1 = [s_{j,i}, a_i)$, and $V_{j,i}^t \; / \; X_{j,i}^t \; / \; W_{j,i}^t \; / \; U_{j,i}^t$ to be the set of points in $V_{j,i} \; / \; X_{j,i} \; / \; W_{j,i} \; / \; U_{j,i}$ with $x$-coordinate in $A_{j,i}^t$ ($t \in \{0,1\}$).

    \begin{figure}[htb]
        \centering
        \includegraphics{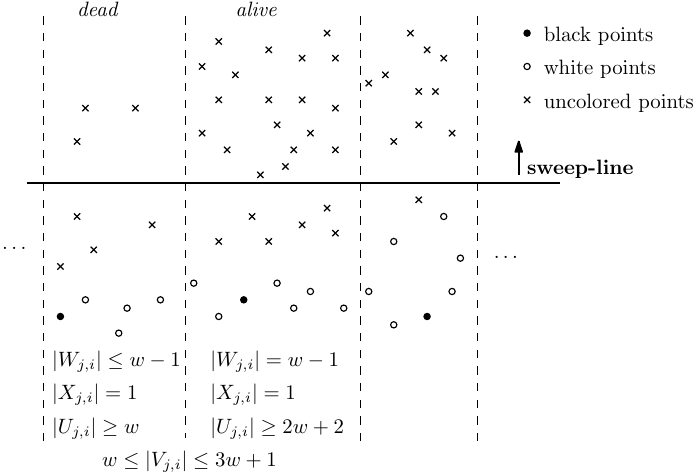}
        \caption{Sweep-line algorithm for $w=7$ after step $j$. The black points are the points in $X_j$, white points are the points in $W_j$ and uncolored points are the points in $U_j$.}
        \label{fig:topless-bottomless-sweepline}
    \end{figure}
    
    \begin{enumerate}[label = (I\arabic*)]
        \item[] We maintain that, after each step $j$, the following invariants to hold for all $i=1,\ldots,l$.
        \item Set $V_{j,i}$ has size $w \leq |V_{j,i}| \leq 3w+1$.\label{inv:1}
        \item The uncolored points in $V_{j,i}$ are the topmost points in $V_{j,i}$.\label{inv:2}
        \item[] \underline{If $A_{j,i}$ is dead, then:}
        \noindent
        \begin{enumerate}[label = (I\arabic*), start = 3]
            \item $V_{j,i}$ contains exactly one black point and at most $w-1$ white points, i.e., $|X_{j,i}| = 1$ and $|W_{j,i}| \leq w-1$.\label{inv:3}
            \item $U_{i,j}$ contains at least $w$ points, i.e., $|U_{j,i}| \geq w$.\label{inv:4}
            \item All points $p \in U_{i,j}$ remain uncolored in all steps $j' \geq j$.\label{inv:5}
        \end{enumerate}
        \item[] \underline{If $A_{j,i}$ is alive, then:}        
        \begin{enumerate}[label = (I\arabic*), start = 6]
            \item $V_{j,i}$ contains exactly one black point and exactly $w-1$ white points, i.e., $|X_{j,i}| = 1$ and $|W_{j,i}| = w-1$.\label{inv:6}
            \item If $|V_{j,i}| \geq 2w$, then there exists a premarked split point $s_{j,i} \in A_{j,i}$ that satisfies the following:
            There exists a $t \in \{0,1\}$ such that $|U_{j,i}^t|=w+1$ and $|V_{j,i}^t| \geq w+1$.\label{inv:7}
        \end{enumerate}        
    \end{enumerate}

    We start with step $j=w$ as follows.
    The set of black points is $X_w = \{p_1\}$, the set of white points is $W_w = \{p_2,\dots,p_w\}$ and $\Rbb = (-\infty, \infty)$ is the partition of $\Rbb$ into $l = 1$ interval.
    This way, all invariants~\ref{inv:1}--\ref{inv:7} are fulfilled, since $|V| \geq 6w+1$ and therefore $|U_{w,1}| \geq 5w+1$ and $A_{w,1}$ is alive.

    Now, suppose that $X_j$, $W_j$ and the partition $\Rbb=A_{j,1} \dot\cup \cdots \dot\cup A_{j,l}$ are given as the result of step $j$.
    In the next step $j+1$, we consider the points $V_{j+1} = V_j \cup \{p_{j+1}\}$.
    Let $A_{j,i'}=[a_{i'-1},a_{i'})$ be the set with $x(p_{j+1}) \in A_{j,i'}$.
    
    First, we construct the new partition $\Rbb = A_{j+1,1} \dot\cup \cdots \dot\cup A_{j+1,l'}$ and the premarked split points.
    We keep all intervals $A_{j,i}$ with $i \neq i'$, i.e., for each $A_{j,i}$ with $i \neq i'$ there exists an interval $A_{j+1,i''} = A_{j,i}$.
    We shall also keep all uncolored, black, and white points in all these intervals $A_{j,i}$ with $i \neq i'$, and thus invariants \ref{inv:1}--\ref{inv:7} are maintained here.
    
    The only changes to intervals and colors of points will be in the interval $A_{j,i'}$ containing the ``new'' point $p_{j+1}$ and the corresponding subset $V_{j,i'}$ of points in $V$.
    Depending on the following cases, we either keep the interval $A_{j,i'}$, find a premarked split point for $A_{j,i'}$, or replace $A_{j,i'}$ by two intervals, sometimes by cutting at its already existing premarked split point.
   
    \begin{enumerate}[label=(\alph*)]
        \item \textit{$A_{j,i'}$ is dead; or $A_{j,i'}$ is alive and $|V_{j,i'}| < 2w-1$.}

        In this case, we set $A_{j+1,i} = A_{j,i}$ for all $i$ and add no split points.
        
        \item \textit{$A_{j,i'}$ is alive and $|V_{j,i'}| = 2w-1$.}

        Let $V_{j,i'} \cup \{p_{j+1}\} = \{q_1,\dots,q_{2w}\}$ with $x(q_1) < \cdots < x(q_{2w})$ and let $b = x(q_w)$ and $b' = x(q_{w+1})$.
        In other words, $b$ and $b'$ are the $x$-coordinates of the two horizontally middle points in $V_{j,i'} \cup \{p_{j+1}\}$.
        We distinguish two cases.
        
        \textbf{Case 1:} There exists an $a \in (b,b']$ such that there are at least $w+1$ uncolored points in $V$ with $x$-coordinate in $A^0 := [a_{j,i'-1}, a)$ and at least $w+1$ uncolored points in $V$ with $x$-coordinate in $A^1 := [a, a_{j,i'})$.
        Then, define $A_{j+1,i'} = A^0$ and $A_{j+1, i'+1} = A^1$ (for such an $a$), i.e., we cut the interval $A_{j,i'}$ at $a$ and replace it by two new intervals. 

        \textbf{Case 2:} Otherwise, assume without loss of generality that for each $a \in [b,b')$ there are at most $w$ uncolored points in $V$ with $x$-coordinate in $[a_{j,i'-1},a)$.
        We set $A_{j+1,i} = A_{j,i}$ for all $i$ but add a premarked split point $s_{j,i'}$, which we define in the following.
        As $A_{j,i'}$ is alive, i.e., at least $2w+2$ uncolored points in $V$ have their $x$-coordinate in $A_{j,i'}$, it follows that there exists an $s \in (b',a_{j,i'}]$ such that there are exactly $w+1$ uncolored points in $V$ with $x$-coordinate in $A^0 := [a_{j,i'-1}, s)$.
        Observe that there are at least $w+1$ points in $V_{j+1,i'}$ with $x$-coordinate in $A^0$.
        Therefore, $s$ satisfies both conditions in \ref{inv:7} of being a split point, and we add $s$ as a premarked split point.

        \item \textit{$A_{j,i'}$ is alive and $|V_{j,i'}| > 2w-1$.}

        Then by~\ref{inv:7}, there exists a premarked split point $s_{j,i'} \in A_{j,i'}$ such that, without loss of generality, $|U_{j,i}^0| = w+1$ and $|V_{j,i}^0| \geq w+1$.
        
        \textbf{Case 1:} If $x(p_{j+1}) \in A_{j,i'}^0$ or $|V_{j,i'}^1 \cup \{p_{j+1}\}| < w$ then we set $A_{j+1,i} = A_{j,i}$ for all $i$.
        
        \textbf{Case 2:} If $x(p_{j+1}) \in A_{j,i'}^1$ and $|V_{j,i'}^1 \cup \{p_{j+1}\}| = w$, then we split the interval at the split point $s = s_{j,i'}$, i.e., define $A_{j+1,i'} := A_{j,i'}^0 = [a_{j,i'-1}, s)$ and $A_{j+1,i'+1} := A_{j,i'}^1 = [s,a_{j,i'})$.
        Observe that $|U_{j+1,i'}| = w+1 < 2w+2$, and hence $A_{j+1,i'}$ is dead but fulfills \ref{inv:4}.
        Also note that $|V_{j+1,i'+1}| = w$ and hence $A_{j+1,i'+1}$ fulfills \ref{inv:1}.
    \end{enumerate}
    Now, we construct the sets $X_{j+1}$ and $W_{j+1}$.
    If we do not split the interval $A_{j,i'}$, then we set $X_{j+1}=X_j$ and $W_{j+1}=W_j$.
    Note that this way, invariants \ref{inv:1}--\ref{inv:7} hold for $A_{j,i'}$ (and also for all other intervals).
    Otherwise, we have split the interval $A_{j,i'}$ into two intervals $A^0 = [a_{j,i'-1},a)$ and $A^1 = [a, a_{j,i'})$ for some $a$, as defined previously.
    For these intervals, we define $V^t \; / \; X^t \; / \; W^t \; / \; U^t$ to be the set of points in $V_{j,i'} \cup \{p_{j+1}\} \; / \; X_{j,i'} \; / \; W_{j,i'} \; / \; U_{j,i'}$ with $x$-coordinate in $A^t$ (for $t \in \{0,1\}$).
    Note that in this case $A_{j,i'}$ was alive and hence by \ref{inv:6} contained $w-1$ white points.
    Thus $|W^t| \leq w-1$ for $t = 0,1$.
    Let $t \in \{0,1\}$ be arbitrary.
    Then, one of the following cases occurs.
    \begin{itemize}
        \item[] \textbf{Case 1:} $|V^t| = w$ and $|U^t| \geq 2w+2$.

        That is, $V^t$ is alive.
        If there is no black point in $V^t$, then there exists an uncolored point $p$ in $V^t$, since $|W^t| \leq w-1$ and $|V^t| = w$.
        We color $p$ black and all other uncolored points in $V^t$ white.
        Otherwise, if there is a black point in $V^t$, then we color all uncolored points in $V^t$ white.
        As $|V^t| = w$, it follows that \ref{inv:6} holds.
        
        \item[]
        \textbf{Case 2:} $|V^t| = w$ and $w+1 \leq |U^t| < 2w+2$, or

        \textbf{Case 3:} $|V^t| > w$ and $|U^t| = w+1$.

        If there is no black point in $V^t$, then there exists an uncolored point in $V^t$, as $|V^t| \geq w$ and $|W^t| \leq w-1$.
        Let $p$ be the bottommost uncolored point in $V^t$.
        We color $p$ black but leave all other uncolored points uncolored.
        Observe that this interval $A^t$ is dead, as $|U^t| < 2w+2$, and satisfies invariants \ref{inv:3} and \ref{inv:4}.
        In fact, we still have at least $w$ uncolored points as required by \ref{inv:4}, since $|U^t| \geq w+1$ and we only colored one point black.
    \end{itemize}
    
    This completes step $j+1$.
    After step $n = |V|$, we have considered all points in $V$.
    We show that $X := X_n$ is a $14$-shallow hitting set in $\Hcal(V,\Rbl,m)$ and a $7$-shallow set in $\Hcal(V,\Rtl,m)$.

    \begin{claim}
    \label{claim:bottomless-hitting-2}
        $X$ is hitting in $\Hcal(V, \Rbl, m)$.
    \end{claim}
    \begin{claimproof}
        Let $R = [a,b] \times (-\infty,c]$ be a bottomless rectangle with $|R \cap V| = m$.
        Then, there exists a $j$ such that the points in $V_j$ with $x$-coordinate in $[a,b]$ are exactly the points in $R \cap V$.
        We have $|R \cap V| = m \geq 6w+1$, and $|V_{j,i}| \leq 3w+1$ for each $V_{j,i}$ by~\ref{inv:1}.
        Therefore, there exists a $V_{j,i'}$ with $V_{j,i'} \subseteq R \cap V$.
        Since $V_{j,i'}$ contains at least one black point by \ref{inv:3} and \ref{inv:6}, $R \cap V$ contains a black point too and $X$ is hitting.
    \end{claimproof}

    \begin{claim}
    \label{claim:shallowness-of-intervals-2}
        $|X \cap V_{j,i}| \leq 2$ for every $V_{j,i}$, and $|X \cap V_{j,i}| = 1$ if $|V_{j,i}|=w$.
    \end{claim}
    \begin{claimproof}
        If $V_{j,i}$ is dead, then by \ref{inv:6} we have $|X_{j,i}| = 1$ and by \ref{inv:5} all uncolored points in $V_{j,i}$ remain uncolored.
        Thus $|X \cap V_{j,i}| = 1$ in this case.
        
        If $V_{j,i}$ is alive, then by \ref{inv:3} we have $|X_{j,i}| = 1$ and $|W_{j,i}| = w-1$.
        Therefore, if $|V_{j,i}|=w$, then $|X \cap V_{j,i}| = 1$.
        Moreover, whenever an uncolored point $q \in V_{j,i}$ is colored black, then every other uncolored point $p$ in $V_{j,i}$ is either colored white (if $p$ lies in an interval that is alive in step $j+1$) or remains uncolored throughout the algorithm (if $p$ lies in an interval that is dead in step $j+1$).
        Therefore, $|X \cap V_{j,i}| \leq 2$ for every $V_{j,i}$.
    \end{claimproof}

    \begin{claim}
    \label{claim:bottomless-shallow-2}
        $X$ is $14$-shallow in $\Hcal(V, \Rbl, m)$.
    \end{claim}
    \begin{claimproof}
        Let $R$ be a bottomless rectangle that contains $m$ points of $V$, i.e., $|R \cap V| = m$ and let $p=p_j$ be the topmost point in $R \cap V$ (for some $j$).
        We consider the sweep-line algorithm right after step $j$.
        We have $|R \cap V| = m \leq 6w+6$ and each $V_{j,i}$ contains at least $w$ points.
        Each $V_{j,i}$ contains at most two points of $X$ by \cref{claim:shallowness-of-intervals-2}, and exactly one point of $X$ if $|V_{j,i}|=w$.
        Observe that in worst case, every $V_{j,i}$ that intersects $R$ has size $w+1$ and two points in $X$.
        Then, $R$ intersects at most $7$ such sets $V_{j,i}$ and therefore, $|R \cap X| \leq 7 \cdot 2 = 14$ and $X$ is $14$-shallow.
    \end{claimproof}

    \begin{figure}[htb]
        \centering
        \includegraphics[width=\textwidth]{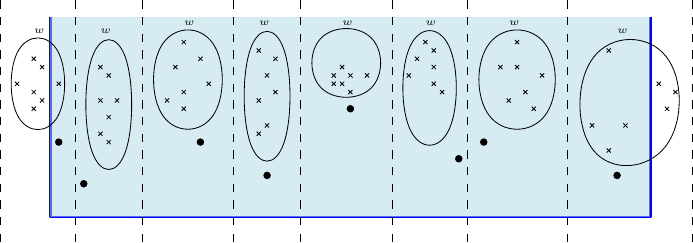}
        \caption{If a topless rectangle contains at least $8$ points of $X$, then it has size at least $2+6(w+1) > 6w+6 \geq m$.}
        \label{fig:topless-shallow}
    \end{figure}
    
     \begin{claim}
     \label{claim:topless-shallow}
         $X$ is $7$-shallow in $\Hcal(V, \Rtl, m)$.
     \end{claim}
     \begin{claimproof}
        Consider the sweep-line algorithm after the last step $j=n$.
        By invariants \ref{inv:3} and \ref{inv:6}, each $V_{n,i}$ has exactly one point in $X = X_n$.
        By invariants \ref{inv:4} and \ref{inv:6}, each $V_{n,i}$ has at least $w$ uncolored points, which by \ref{inv:2} are the topmost points in $V_{n,i}$.
        See \cref{fig:topless-shallow}.
        Consider a topless rectangle $R = [a,b] \times [c, \infty)$ of size $|R \cap V| = m \leq 6w+6$.
        If an interval $A_{n,i}$ satisfies $A_{n,i} \subseteq [a,b]$ and the black point $p \in X_{n,i}$ is contained in $R$, then there are at least $w$ uncolored points in $R \cap V_{n,i}$.
        Moreover, there are at most two intervals $A_{n,i}$ with $\emptyset \neq A_{n,i} \cap [a,b] \subsetneq A_{n,i}$ and therefore at most two sets $V_{n,i}$ with $\emptyset \neq V_{n,i} \cap R \subsetneq V_{n,i}$.
        Assume that $|R \cap X| \geq 8$, then $|R \cap V| \geq (8-2) \cdot (w+1) + 2 > 6w+6 \geq m$, see \cref{fig:topless-shallow}.
        This is the desired contradiction.
     \end{claimproof}

     By \cref{claim:bottomless-hitting-2,claim:bottomless-shallow-2,claim:topless-shallow}, $X$ is a $14$-shallow hitting set in $\Hcal(V,\Rbl,m)$ and a $7$-shallow set in $\Hcal(V,\Rtl,m)$, which concludes the proof.
\end{proof}

Having \cref{lem:bottomless-topless-lemma} in place, we can quickly derive the full theorem.

\begin{theorem}
\label{thm:bottomless-topless-shallow-hitting-set}
    Let $\Rbl \cup \Rtl$ be the range family of all bottomless and topless rectangles in $\Rbb^2$ and $m$ be a positive integer.
    Then for any finite point set $V \subset \Rbb^2$ the hypergraph $\Hcal(V,\Rbl \cup \Rtl,m)$ admits a $21$-shallow hitting set $X \subseteq V$.
\end{theorem}
\begin{proof}
    By \cref{lem:bottomless-topless-lemma}, there exists a set $Y$ that is a $14$-shallow hitting set of $\Hcal(V,\Rbl,m)$ and a $7$-shallow set in $\Hcal(V,\Rtl,m)$.
    Symmetrically, there exists a set $Z$ that is a $14$-shallow hitting set of $\Hcal(V,\Rtl,m)$ and a $7$-shallow set in $\Hcal(V,\Rbl,m)$.
    Then, $X = Y \cup Z$ is a $21$-shallow hitting set of $\Hcal(V,\Rbl \cup \Rtl,m)$.
\end{proof}

\begin{theorem}
\label{thm:unit-height-rectangles}
    Let $\Ruh$ be the range family of all unit-height axis-aligned rectangles in $\Rbb^2$ and $m$ be a positive integer.
    Then, for every finite point set $V \subset \Rbb^2$ the hypergraph $\Hcal(V, \Ruh, m)$ admits a $63$-shallow hitting set $X \subseteq V$.
    Moreover, $m_{\Ruh}(k) \leq 2m_{\Rbl\cup\Rtl}(k)-1 \leq 12k - 7$ for the range family $\Ruh$.
\end{theorem}
\begin{proof}
    Let $m$ be a positive integer and $V \subset \Rbb^2$ be a finite point set.
    Define $m' = \lceil m/2 \rceil$.
    For every integer $a \in \Zbb$, let $H_a = \Hcal(V_a, \Rbl \cup \Rtl, m')$ be the range capturing hypergraph induced by the range family of all bottomless and topless rectangles, where $V_a$ is the set of all points $p$ in $V$ with $a \leq y(p) < a+1$.
    By \cref{thm:bottomless-topless-shallow-hitting-set}, every $H_a$ admits a $21$-shallow hitting set $X_a$.
    Then, $X=\bigcup_{a \in \Zbb} X_a$ is a $63$-shallow hitting set in $\Hcal(V,\Ruh,m)$, which can be seen as follows.
    Every unit-height rectangle induces a topless rectangle $R_t$ in $H_a$ and a bottomless rectangle $R_b$ in $H_{a+1}$ (for some $a$).
    Then, at least one of $R_t$ and $R_b$ contains at least $\lceil m/2 \rceil = m'$ points of $V$, without loss of generality $R_t$.
    Therefore, $R_t$ contains a point of $X_a$ and hence, $X$ is hitting.
    Since $m \leq 2m'$, the topless rectangle $R_t$ can be covered with at most two topless rectangles of size $m'$ of $H_a$, and $R_b$ can be covered with at most one bottomless rectangle of size $m'$ of $H_{a+1}$.
    As each of these three rectangles contains at most $21$ points of $X$, we conclude that $X$ is $t$-shallow for $t=3 \cdot 21 = 63$.

    Using the same argument, it is not difficult to see that $m_{\Ruh}(k) \leq 2m_{\Rbl\cup\Rtl}(k)-1$.
    Let $m=2m_{\Rbl\cup\Rtl}(k)-1$ and let $H=\Hcal(V, \Ruh, m)$ be the range capturing hypergraph induced by all unit-height rectangles.
    Let $m' = \lceil m/2 \rceil = m_{\Rbl\cup\Rtl}(k)$.
    For every $a \in \Zbb$, color each $H_a = \Hcal(V_a, \Rbl \cup \Rtl, m')$ polychromatically with $k$ colors with respect to bottomless and topless rectangles $\Rbl \cup \Rtl$.
    This polychromatic coloring exists by \cref{thm:bottomless-topless-m(k)} and since $m'=m_{\Rbl\cup\Rtl}(k)$.
    Then, every unit-height rectangle $R$ induces a topless rectangle $R_t$ in $H_a$ of size at least $m'$ or a bottomless rectangle $R_b$ in $H_{a+1}$ of size at least $m'$ (for some $a \in \Zbb$).
    Since $R_t$ (respectively $R_b$) contains points of all colors, so does $R$.
    Therefore, each unit-height rectangle with $m$ points contains points of all colors and we have found a polychromatic $k$-coloring of $\Hcal(V,\Ruh,m)$.
\end{proof}

\section{Conclusions}
\label{sec:conclusions}

In this paper, we extended the list of range families $\Rcal$ for which the corresponding uniform range capturing hypergraphs admit shallow hitting sets.
This in particular implies that $m_{\Rcal}(k) = O(k)$ for that family $\Rcal$, while $m_\Rcal(k) \geq k$ always holds.
In view of \cref{oque:natural-family}, it would be interesting to investigate further range families $\Rcal$ for which $m_{\Rcal}(k) < \infty$ is known, as to whether they admit shallow hitting sets.
The current state of the art (for a selection of range families) is summarized in \cref{tab:state-of-the-art}.

\begin{table}
    \renewcommand{\arraystretch}{1.3}
    \begin{tabularx}{\textwidth}{cXXX}
        \toprule
        & range family $\Rcal$ & $t$-shallow hitting sets exist & $m_\Rcal(k)$ \\
        \midrule
        \ref{1:known-strips} & axis-aligned strips in $\Rbb^d$ & \textbf{Yes} for $t \geq 3\euler d(1+o(1))$ \newline \hspace*{1pt} \hfill {\small (\cref{thm:union-of-strips})} & $O_d(k)$ \hfill {\small (\cref{cor:union-of-strips-2})} \\

        \ref{2:known-bottomless} & bottomless rectangles in $\Rbb^2$ & \textbf{Yes} for $t \geq 10$ \newline \hspace*{1pt} \hfill {\small (\cref{thm:bottomless-shallow-hitting-set})} & $\leq 3k-2$ \hfill \cite{ACCCHHKLLMRU13} \\

        \ref{3:known-halfplanes} & half-planes in $\Rbb^2$ & \textbf{Yes} for $t \geq 2$ \hfill \cite{SY12} & $\leq 2k-1$ \hfill \cite{SY12} \\

        \ref{4:known-squares} & axis-aligned squares in $\Rbb^2$ & \textbf{Open} & $O(k^{8.75})$ \hfill \cite{AKV17} \\

        \ref{5:known-bottomless-topless} & bottomless and topless rectangles in $\Rbb^2$ & \textbf{Yes} for $t \geq 21$ \newline \hspace*{1pt} \hfill {\small (\cref{thm:bottomless-topless-shallow-hitting-set})} & $\leq 6k-3$ \hfill {\small (\cref{thm:bottomless-topless-m(k)})} \\

        \ref{6:known-polygon} & translates of a convex polygon in $\Rbb^2$ & \textbf{Open} & $O(k)$ \hfill \cite{GV09} \\

        \ref{7:known-triangle} & homothets of a triangle in $\Rbb^2$ & \textbf{Open} & $O(k^{4.09})$ \hfill \cite{KP15} \\

        \ref{8:known-octant} & translates of octants in $\Rbb^3$ & \textbf{No} \hfill \cite{CKMPUV23} & $O(k^{5.09})$ \hfill \cite{KP15}\\
        \bottomrule
    \end{tabularx}
    \caption{Shallow hitting sets and polychromatic colorings for range capturing hypergraphs.}
    \label{tab:state-of-the-art}
\end{table}

Let us also mention that Keszegh and P{\'a}lv{\"o}lgyi~\cite{KP19} define a $k$-coloring $c \colon V \to \{1,\ldots,k\}$ of a hypergraph $H = (V,E)$ to be \emph{$t$-balanced} if for any two colors $i,j \in \{1,\ldots,k\}$ and any hyperedge $e \in E$ we have $|\{v \in e \mid c(v) = i\}| \leq t \cdot (|\{v \in e \mid c(v) = j\}|+1)$, i.e., in each hyperedge there are at most roughly $t$ times as many vertices of color $i$ than of color $j$.
They show that if a (shrinkable) range family admits $t$-shallow hitting sets then it also allows for $t$-balanced $k$-colorings for every $k$.
And conversely, if we have $t$-balanced $k$-colorings for every $k$, then we have $t^2$-shallow hitting sets.
Thus, \cref{thm:bottomless-shallow-hitting-set} for example gives that every range capturing hypergraph $\Hcal(V,\Rbl,m)$ for bottomless rectangles admits a $10$-balanced $k$-coloring for every $k \geq 2$.

Let us also mention Beck's three permutation conjecture made in 1987, which asks whether the discrepancy of the hypergraph $\Hcal(V,\Rst)$ induced by all axis-aligned strips in the case of $d=3$ dimensions is $O(1)$.
The conjecture has been refuted in 2012~\cite{NNN12}, and although discrepancy is defined via a $2$-coloring being ``balanced'' in a certain way, we do not see an immediate connection to balanced or polychromatic colorings.

\bibliographystyle{plainurl}
\bibliography{literatur}

\end{document}